\documentclass[preprint,smallextended]{svjour3q}  

\smartqed  

\usepackage{graphicx}

\usepackage{newcent}
\usepackage{amssymb}

\def\makeheadbox{}
\newcommand{\K}{\mathbb{K}}       
\newcommand{\R}{\mathbb{R}}       
\newcommand{\C}{\mathbb{C}}       
\newcommand{\HH}{\mathbb{H}}    

\newcommand{\cat}{\mathop{\mathrm{cat}}}     
\newcommand{\id}{\mathrm{id}} 
\newcommand{\diag}{\mathop{\mathrm{diag}}}   
\newcommand{\grad}{\mathop{\mathrm{grad}}}   
\newcommand{\Tr}{\mathop{\mathrm{Tr}}}  
 
\newcommand{\lcp}{\mathop{\mathrm{l.c.p.}}}

\begin{document}

\title{Trace map, Cayley transform and LS category of Lie groups
\thanks{Partially supported by FEDER and Research Project MTM2008-05861 MICINN Spain}
}

\titlerunning{Trace map, Cayley transform and LS category of Lie groups}        

\author{%
A.~G\'{o}mez-Tato
\and
E.~Mac\'{\i}as-Virg\'{o}s 
\and 
M.~J.~Pereira-S\'{a}ez 
}


\institute{%
		A.~G\'{o}mez-Tato \at
		Institute of Mathematics, Department of Geometry and 			Topology\\ University of Santiago de Compostela, 15782-			Spain \\
              	\email{antonio.gomez.tato@usc.es} 
\and 
		E.~Mac\'{\i}as-Virg\'{o}s \at  
		\email{quique.macias@usc.es}  
\and 
              	M.~J.~Pereira-S\'{a}ez \at 
              	\email{mariajose.pereira@usc.es}           
                            }

\date{}

\maketitle

\begin{abstract}
The aim of this paper is to use the so-called Cayley transform  to compute the LS category of  Lie groups and homogeneous spaces  by giving explicit categorical open coverings.  When applied to $U(n)$, $U(2n)/Sp(n)$ and $U(n)/O(n)$  this method is simpler than those formerly known.   We also show that the Cayley transform  is related to  height functions in Lie groups, allowing to give a local linear model of the set of critical points. As an application we give an explicit covering of $Sp(2)$ by  categorical open sets. The obstacles to generalize  these results to $ Sp(n)$ are  discussed.
\end{abstract}

\setcounter{tocdepth}{2} 
\tableofcontents

\section{Introduction}
\label{intro}

Lusternik-Schnirelmann category is a homotopical invariant that has been  widely studied \cite{TANRE,JAMES}.  For a topological space $X$, the LS category $\cat X$ is defined as the minimum number (minus one) of categorical open sets which are needed to cover $X$ (an open set is categorical when it is contractible in $X$).

Unfortunately the LS category  is very difficult to compute. For instance, while the result $\cat Sp(2)=3$ was proven  by P.~Schweitzer  in 1965 \cite{PAUL}, it was not until 2002 that  $\cat Sp(3)=5$ appeared in \cite{TATO}, see also \cite{IWASEMIMURA}.  In general, the algebraic techniques involved are highly elaborated.  It is then of interest to introduce more elementary methods.

The main idea of this paper is to compute the LS category of some Lie groups and homogeneous spaces of the orthogonal type by means of the so-called Cayley transformation. This will give proofs ---which are simpler than the original ones---, of $\cat U(n)=n$ (by  W.~Singhof \cite{SINGHOF}), $\cat U(n)/Sp(n)=n$ and  $\cat U(n)/O(n)=n$ (by M.~Mimura and K.~ÊSugata \cite{MIMURASUGATA}). 

Our method is closely related to  Morse theory on Lie groups. Classically, the  functions that use to be considered are ``height" or ``distance" as in \cite{DUAN,VD,VOLKOZ}. On a matrix Lie group $G$  these functions are, up to a constant, of the form $h_X(A)=\Re\Tr(XA)$,  the real part of the trace, for some matrix $X$, a fact which allows to explicitly describe the Bott-Morse structure of these functions. We shall prove that the Cayley transform serves to linearize the gradient flow of $h_X$ and to give local charts for the set of critical points. These results generalize those of K.~Y. Volchenko and A.~N. Kozachko (\cite{VOLKOZ}, see also \cite{VD}).

Let us remember that in a compact manifold the LS category (plus one) is a lower bound for the number of critical points of any smooth function (Morse or not).
The reason is that --roughly speaking-- for each critical point the gradient flow defines a categorical open set. What is nice in our setting is that this flow is given by the contraction associated to the Cayley map. 
 
Moreover, the Morse interpretation above allows us to give an explicit covering of $ Sp(2)$ by four categorical open sets,  a result that completes the abstract proof by  Schweitzer  \cite{PAUL}.

At the end of the paper we explain how the generalization of our results to the symplectic group $ Sp(n)$  depends on the computation of the so-called left eigenvalues of a quaternionic matrix \cite{ZHANG2},  a topic about which very little is known, out of the case $n=2$. 
 
We hope that the ideas presented here will deserve further attention.


\section{The Cayley transform}\label{CAYLEYTRANSFORM}

The classical Cayley transform was introduced by A.~Cayley in 1846 \cite{CAYLEY}, as a way to express an orthogonal transformation by means of skew-symmetric coordinates. It is given by
$$c(X)={I-X\over I+X}.$$ 
This map is defined for all matrices having their eigenvalues different from $-1$ and equals its own inverse, $c^2=\id$. It can be thought as a generalization of the stereographic projection.

Its basic properties appear in \cite{POSTNIKOV}, see also \cite{WEYL}.

In order to obtain a categorical covering of the orthogonal Lie groups we shall introduce in the next paragraphs a convenient generalization of the classical Cayley map.

\subsection{Preliminaries}

Let  the algebra $\K$ be either $\R$ (reals), $\C$ (complex) or $\HH$ (quaternions).
We say that the matrix $A\in\mathcal{M}(n,\K)$ is {\em orthogonal} if $AA^*=\id$, where $A^*=\bar{A^t}$ is the conjugate transpose. Such a matrix can be identified with a (right) $\K$-linear map $\K^n \to \K^n$ preserving the product $\langle v,w\rangle=v^*w$.
Let us denote  by
$G=O(n,\K)$ the  Lie group of orthogonal matrices. 
Depending on $\K$ this group corresponds to the orthogonal group $O(n)$, the unitary group $U(n)$ or the symplectic group $ Sp(n)$.

\begin{remark} The Cayley transform maps a   classic orthogonal Lie group like $G=U(n)$ or $Sp(n)$    into its Lie algebra ${\mathfrak g}$ of skew-hermitian matrices. In fact, suppose that $A$ is a unitary or symplectic matrix. It can be diagonalized,
$A=UDU^*$, to a {\em complex} diagonal matrix $D=\diag(\lambda_1,\dots,\lambda_n)$ \cite{BRENNER}. Then $$c(A)=U\diag(\pi(\lambda_1),\dots,\pi(\lambda_n))U^*,$$ where $\pi$ is the stereographic projection $\pi\colon S^1\backslash \{-1\}\to \R$.
\end{remark}

\subsection{Generalized Cayley transform}

Let $A\in O(n,\K)$ be an orthogonal matrix, where $\K$ is $\R$, $\C$ or $\HH$.

\begin{definition}  Let us denote by $\Omega(A)\subset \mathcal{M}(n,\K)$ the open set  of matrices $X$ such that $A+X$ is invertible. The \emph{Cayley transform centered at    $A$} is the map $$c_A\colon \Omega(A) \to \Omega(A^*)$$ given by
$$c_A(X)=(I-A^*X)(A+X)^{-1}.$$
\end{definition}

The classical Cayley map corresponds to $A=I$.
As we shall see in the next Proposition, the application $c_A$ is well defined  and it is invertible, with
 $c_A^{-1} = c_{A^*}$.

\begin{proposition}\label{DEFIN}
 If $X\in \Omega(A)$ then
\begin{enumerate}
\item
$c_A(X)=(A+X)^{-1}(I-XA^*)$;
\item
the inverse matrix of $A^*+c_A(X)$ is $(1/2)(A+X)$;
\item
if $X\in\Omega(A)$ then $c_A(X)\in\Omega(A^*)$;
\item
$c_A$ is a diffeomorphism, with $c_A^{-1} = c_{A^*}$.
\end{enumerate}
\end{proposition}

\begin{proof}
(1)
It suffices to verify that   $(A+X)(I-A^*X)= (I-XA^*)(A+X)$, which is immediate because $AA^*=\id$.
For (2) we compute
$$\left(A^*+(I-A^*X)(A+X)^{-1}\right)(1/2)(A+X)=(1/2)\left(A^*A+A^*X+I-A^*X\right)= I .$$
Part (3) comes immediately from (2).
Finally, by using (2) we obtain that 
\begin{eqnarray*}
          (c_{A^*}\circ c_A)(X) &=&  \\
          (I-Ac_A(X))(1/2)(A+X) &=&  \\
          \left(I-A(I-A^*X)(A+X)^{-1}\right)(1/2)(A+X) &=&  \\
          (1/2)\left((A+X)-A+AA^*X\right) &=& X.
        \end{eqnarray*}
\qed
\end{proof} 
We shall need the following interesting properties, which are easy to prove:

\begin{proposition}\label{PROP2}
Let $X\in \Omega(A)$. Then
\begin{enumerate}
\item
 $X^* \in \Omega(A^*)$ and $c_{A^*}(X^*)=c_A(X)^*$;
\item
$UXU^*\in\Omega(UAU^*)$ for any matrix $U\in O(n,\K)$ and
$$c_{UAU^*}(UXU^*)=Uc_A(X)U^*;$$
\item
if the matrix $X$ is invertible then $X^{-1} \in \Omega(A^*)$
and $$c_{A^*}(X^{-1}) = -Ac_A(X)A .$$
\end{enumerate}
\end{proposition}

\subsection{Categorical open sets}
The results in this paragraph show that the domain of the Cayley transform in an orthogonal group is contractible.

The Lie algebra of  $G=O(n,\K)$  is formed by the skew-symmetric (resp. skew-hermitian) matrices,
$$\mathfrak{g}=\mathfrak{o}(n,\K)=\{X\in\mathcal{M}(n,\K) \colon X+X^*=0\}.$$
As a vector space $\mathfrak{g}= T_IG$, so the tangent space at any other point $A\in G$ is
$$T_AG =L_A(T_IG)=\{Y\in\mathcal{M}(n,\K)\colon A^*Y+ Y^*A=0\}.$$

\begin{proposition}\label{noautovreales}
Let $X\in\mathfrak{o}(n,\K)$ be a skew-symmetric (resp. skew-hermitian) matrix. Then $X$ has not real eigenvalues different from zero.
\end{proposition}

\begin{proof} 
Suppose that there exists $t\in\R$ such that $Xv=vt$ for some
$v\in\K^n$, $v\neq 0$. Then $v^{*}Xv=v^{*}vt=\vert v\vert^{2}t$ is a real number and, consequently,
$$v^*Xv = (v^*Xv)^* = v^*X^*v = v^*(-X)v =-v^*Xv.$$ Therefore $v^{*}Xv$ is  null. \textit{i.e.} $
\vert v \vert^2 t=0$ hence $t=0$.
\qed\end{proof}

\begin{corollary}\label{TANCONT} The real vector space  
 $T_{A}G$ of the matrices $Y$ such that
$A^*Y+Y^*A=0$ is contained in $\Omega(A)=\{Y\in \mathcal{M}(n,\K) \colon \, A+Y \mathrm{\ is\ invertible}\}$.
\end{corollary}

\begin{proof}
If $A+Y$  is not invertible then there exists $v\neq 0$ such that $Yv=-Av$ so $A^*Yv=-v$. This means that the skew-symmetric matrix $A^*Y$ has $-1$ as an eigenvalue,  contradicting Proposition \ref{noautovreales}.\qed \end{proof}

Let $G=O(n,\K).$ We shall denote by $\Omega_G(A)$ the open subset $\Omega(A)\cap G \subset G.$

\begin{theorem} \label{CONTRACTIBLE}
The generalized Cayley transform $c_A$ maps diffeomorphically  $\Omega_G(A)$ onto  $T_{A^*}G$, with $c_A(A)=0$. As a consequence, the open set $\Omega_G(A)$ is contractible.
\end{theorem}

The proof is an immediate consequence of Propositions \ref{DEFIN}, \ref{PROP2}  , \ref{noautovreales}
and Corollary~\ref{TANCONT}.

\section{Bott-Morse functions on Lie groups}\label{MORSE}
There is a deep relationship between the Cayley transform and  Morse theory in Lie groups.

In this Section \ref{MORSE} we prove   (Proposition \ref{FLUJO}) that one can integrate the gradient flow of any height function $h_X\colon G \to \R$ by applying the Cayley transform $c_{A^*}$ to a simple curve 
in $T_{A^*}G$, provided that  $c_{A^*}(0)=A$ is a critical point. As a consequence we give a local model for the set $\Sigma$ of critical points of $h_X$ (Theorem \ref{MODEL}).

\subsection{Critical points of a height function}
Let $G=O(n,\K)$ be an orthogonal group embedded in the euclidean space $E=\mathcal{M}(n,\K)$. The euclidean metric is given by
$\langle A,B\rangle =\Re\Tr(A^*B)$, the real part of the trace. I.~A. Dynnikov and A.~P. Veselov \cite{VD} and Volchenko and Kozachko \cite{VOLKOZ} studied the height functions on $G$ (height with respect to some hyperplane), while H.~Duan \cite{DUAN} studied the distance  functions on $G$ (distance to a given point). Up to a constant, both classes of functions are given by the formula 
$$h_X(A)=\Re\Tr(XA),$$ for some matrix $X\in E$, as it is easy to prove. 

\begin{example} \label{ptos criticos diagonales}
When 
$X=\diag(\lambda_1,\dots,\lambda_n)$ is a positive real diagonal matrix, with $\lambda_1<\cdots <\lambda_n$, the function $h_X$ is a perfect Morse function, whose  critical points  are the diagonal matrices 
 $$\diag(\varepsilon_1,\dots,\varepsilon_n), \quad \epsilon_k=\pm 1.$$
This result is proven in  \cite{DUAN} and \cite{VD},
see also \cite{MJP}.
\end{example}

\begin{example}\label{EJFRANK}
On the other hand, the case  $X=tI$, $t\in\R$, was first studied by T.~Frankel in \cite{FRANKEL}. This time, the height map  is a Bott-Morse function, invariant by the adjoint action. The set $\Sigma(n)$ of critical points is formed by the matrices $A$ such that $A^2=I$. 
\end{example} 


 The two preceding examples are particular cases of the next Theorem, that we suppose  is more or less folk. It gives a general description of the set of critical points of an arbitrary height function $h_X$.

 A direct computation shows that the gradient of $h_X$ on $G$ is given by
$$(\grad h_X)_A=\frac{1}{2}(X^*-AXA).$$
Moreover, if 
 $A\in G$ is a critical point, then the Hessian operator $(H h_x)_A\colon T_AG \to T_AG$ is given by
$$(H h_X)_A(U)=-\frac{1}{2}(AXU+UXA), \quad U\in T_AG.$$




We keep the notation of Example \ref{EJFRANK}.

\begin{theorem}\label{GENERICO}
 Let $X\in E=\mathcal{M}(n,\K)$ be an arbitrary matrix. Let $n_0$ be the dimension of its kernel and let $0< t_1<\dots < t_k$ be the non-null (real)  eigenvalues of $XX^*$, with multiplicities $n_1,\dots,n_k$, that is $n_0+n_1+\cdots+n_k=n$. Then the set  of critical points of the height function $h_X$ is 
 $$\Sigma(h_X)\cong O(n_0,\K)\times\Sigma(n_1)\times\cdots\times\Sigma(n_k) \quad \mathrm{(diffeomorphism)}.$$
\end{theorem} 

\begin{proof}  
 First, if $Y=UDU^*$, then $\Sigma(h_Y)=U\Sigma(h_D)U^*$. On the other hand, if $X=US$ is a polar decomposition ($U$ orthogonal, $S$ hermitian) then $\Sigma(h_X)=\Sigma(h_S)U^*$. Finally, if $C=\diag(-I_p,+I_q)$ then $\Sigma(h_Y)=C\Sigma(h_{YC})$. These properties prove that we can restrict ourselves to the case where $X$ is given by diagonal blocks
  $$
\left(\begin{array}{c|c|c|c}0 &&  &   \\\hline
&t_1I &   &  \\\hline 
& & \ddots &  \\
 \hline  && & t_kI
 \end{array}\right), \quad  0<t_1<\dots<t_k,
 $$
  of size $n_0,n_1,\dots,n_k$. The rest of the proof is a direct computation. 
 \qed\end{proof}

The transformations in the proof of Theorem \ref{GENERICO} preserve the non-degenerate critical points, so we have the following corollary which completely characterizes the height functions which are Morse functions.

\begin{corollary}
The height  function $h_X$ is Morse if and only if the matrix $XX^*$ is invertible and has $n$ different eigenvalues.
\end{corollary}


\subsection{Gradient flow and local model of the critical point set}
Our next Proposition is a generalization of the same result for the classical Cayley transform $c_I$ by  Volchenko and Kozachko \cite{VOLKOZ}.  Following the terminology of these authors we shall call {\em linearization} the process of transforming the gradient flow of $h_X$   in $G$ to a flow in the Lie algebra. 

\begin{proposition}\label{FLUJO} Let $h_X$ be an arbitrary height function on $G=O(n,\K)$ and let $A$ be a critical point.
The solution to the gradient equation 
$$\alpha^\prime = \frac{1}{2}(X^*-\alpha X\alpha)$$ passing through $\alpha(0)\in \Omega_G(A)$
is the image by the generalized Cayley transform $c_{A^*}$
 of the curve in $T_{A^*}G$ defined as
 $$\beta(t)=\exp({-XAt/2})\cdot\beta_0\cdot\exp({-AXt/2}), \quad \beta_0=c_{A}(\alpha(0)).$$
\end{proposition}







\begin{proof}
Notice that the matrices $XA$ and $AX$ are symmetric (hermitian) because 
$X^*=AXA$  and $A^*A=\id$. 

First, from $Ae^{XA}=e^{AX}A$ it follows that $A\beta+(A\beta)^*=0$, that is $\beta\in T_{A^*}G$.

Now, from the definition of $\beta$, it is
\begin{equation}\label{DERBETA}
\beta^\prime =
(-1/2)\left( XA\beta+\beta AX\right).
\end{equation}
Let
$$\alpha=c_{A^*}\circ\beta=(I-A\beta)(A^*+\beta)^{-1},$$
hence
$$\alpha(A^*+\beta)=I-A\beta.$$
Derivation gives
$$\alpha^\prime (A^*+\beta)+\alpha\beta^\prime=-A\beta^\prime$$
that is
$$\alpha^\prime (A^*+\beta)= -(A+\alpha)\beta^\prime.$$

By Proposition \ref{DEFIN}, the inverse of
$A^*+\beta =A^*+c_A(\alpha)$ is $(1/2)(A+\alpha)$, so
\begin{eqnarray}\label{UFF}
\alpha^\prime&=& \nonumber\\ 
(-1/2)(A+\alpha)\beta^\prime(A+\alpha)
&=&\\ \nonumber
(+1/4)(A+\alpha)(XA\beta+\beta AX)(A+\alpha).\nonumber
\end{eqnarray}
Now
$$\beta=c_A(\alpha)=(I-A^*\alpha)(A+\alpha)^{-1}=(A+\alpha)^{-1}(I-\alpha A^*)$$
implies that
$$\beta(A+\alpha)=I-A^*\alpha$$
and 
$$(A+\alpha)\beta=I-\alpha A^*$$
so from Equation (\ref{UFF})
\begin{eqnarray*}
4\alpha^\prime &=&\\(A+\alpha)XA(I-A^*\alpha)+(I-\alpha A^*)AX(A+\alpha)&=&\\
2AXA -2\alpha X \alpha&=&\\
2(X^*-\alpha X \alpha).
\end{eqnarray*}
\qed
\end{proof}


 \begin{remark}
 Indeed, when $X=A=I$, $\beta(t)=\exp({-t})\beta_0$ is the radial contraction to $\beta_0$.
\end{remark}

We now show how the Cayley transform serves to give a local chart for the set of critical points. This result is completely new.

\medskip
Let $h_X(A)=\Re\Tr(XA)$ be an arbitrary height function on the Lie group $G=O(n,\K)$. Let $\Sigma$ be the set of critical points of $h_X$. If $A\in\Sigma$  is a critical point we denote by $S(A)$ the real vector space
 $$S(A)=\{\beta_0\in T_{A^*}G \colon \,XA\beta_0+\beta_0AX=0\}.$$
 \begin{theorem}\label{MODEL}
 The Cayley map
$$c_{A^*}\colon S(A)\to \Sigma\cap\Omega_G(A)$$
is a diffeomorphism.
\end{theorem}

\begin{proof}
We  need to prove that the curve $\beta(t)$ in Proposition \ref{FLUJO} is constant if and only if $\beta_0\in S(A)$. This can be achieved by using Equation (\ref{DERBETA}).
Then, by Proposition \ref{FLUJO}, $c_{A^*}(\beta_0)$ is a critical point if and only if $\beta^\prime(t)=0$ for all $t$, if and only if 
\begin{eqnarray*}
XA\beta_0+\beta_0AX=0.
\end{eqnarray*}
\qed\end{proof}

\begin{example} Suppose $X=I$ and $\K=\C$. Then the critical points of $h_I$ are the matrices $A\in U(n)$ such that $A^2=I$. Such a matrix $A=A^*$ can be diagonalized to $D=\diag(\varepsilon_1,\dots,\varepsilon_n)$, $\varepsilon_k=\pm 1$.
On the other hand, $\beta_0\in T_{A}G$ iff $A\beta_0$ is skew-symmetric, $A\beta_0=-\beta_0^*A$, while $\beta_0\in S(A)$ iff
$A\beta_0+\beta_0A=0$. It follows that $\beta_0=\beta_0^*$.

So, for instance, the identity $I$ and its opposite $-I$ are critical points that are isolated because  $S(\pm I)=0$. On the other hand, let 
$A=\diag(I_p,-I_q)$. Then $\beta_0\in T_AG$ must be of the form
$$\beta_0=\pmatrix{0&V^*\cr V&0\cr},$$
which implies $\dim S(A)=2pq$. This is in fact the dimension of the (critical) orbit of $A$, which is diffeomorphic to the Grasmannian $U(p+q)/(U(p)\times U(q))$. 
\end{example}

\begin{example} Let $X=\diag(q_1,\dots,q_n)$ be a diagonal matrix, with $q_k\neq 0$. Assume that $\vert q_1 \vert <\dots<\vert q_n\vert$. This time the gradient condition $A^*X^*=XA$ implies that a critical point has the form $$A=\diag(
\pm {\vert q_1 \vert /q_1},
\dots, \pm {\vert q_n \vert / q_n}
).$$
Since
$$XA=AX=\diag(\varepsilon_1\vert q_1\vert,\dots,\varepsilon_n\vert q_n\vert), \quad \varepsilon_k=\pm 1,$$
it follows that $S(A)=0$. So all critical points are isolated. \end{example}

\section{Applications to LS category}\label{APPLI}

\subsection{The unitary group $U(n)$}\label{UNITARY}
W.~Singhof proved in \cite{SINGHOF} that the LS category of the special unitary group $SU(n)$ is $n-1$, hence that of $U(n)\cong S^1\times SU(n)$ (diffeomorphism) is $n$. Although he obtained an explicit categorical covering by using the exponential map, his method has the inconvenience that a logarithm branch has to be chosen, hence introducing some technical complexities.
 
With the Cayley transform we obtain at each point $A\in G=U(n)$ a contractible open set $\Omega_G(A)$ diffeomorphic to the Lie algebra $\mathfrak{u}(n)$ (see Theorem \ref{CONTRACTIBLE}). It is then very easy to find an explicit categorical covering of the group by $n+1$ open sets. 

\begin{theorem}\label{CATUNITARY}
$\cat U(n)=n.$
\end{theorem}

\begin{proof} 

Let $X\in U(n)$ be a unitary matrix and  let $z\cdot\id\in U(n)$ be the diagonal matrix $\diag(z,\dots,z)$, where $z\in\C$ is any complex number with $\vert z \vert=1$. Remember that $X\in \Omega(z\cdot\id)$ iff  the matrix $z\cdot\id+X$ is invertible. Let  $\lambda_1,\ldots,\lambda_n$ be the eigenvalues of $X$, then after diagonalizing we obtain:
$$z\cdot\id+X=U\diag(z+\lambda_1,\ldots,z+\lambda_n)U^*, \quad U\in U(n),$$
meaning that $X\in\Omega(z\cdot\id)$ iff $\lambda_i\neq -z, \forall i=1,\ldots,n.$

Let us take $n+1$ different complex numbers $z_0,\dots,z_n,$ with $\vert z_k\vert=1.$ Let  $A_k=-z_k \cdot \id$. Since any matrix $X\in U(n)$  has at most $n$ different eigenvalues   
 there is always some $-z_k$ which is not an eigenvalue of $X$, that is  $X\in \Omega(A_k)$. This proves that $U(n)=\cup^n_{k=0}{\Omega_G(A_k)}$,  
hence
$\cat U(n)\leq n$. On the other hand, as it is well known \cite{JAMES}, a lower bound for the LS category  is  given by the length of the cup product.
The cohomology of the unitary group being 
$$H(U(n))=\Lambda( x_1,x_3,\ldots, x_{2n-1})$$ (see 
\cite[p.~273]{TANRE}), the longest non-null cup product  is $x_1\wedge x_3\wedge\cdots\wedge x_{2n-1}\in H^{n^2}$, hence $n=\lcp \leq \cat U(n).$ Equality follows.
\qed\end{proof} 

\subsection{The symmetric spaces $U(2n)/Sp(n)$ and $U(n)/O(n)$}\label{SYMMETRIC}

\begin{theorem} $\cat U(2n)/Sp(n)=n.$
\end{theorem}
The following proof is also a simplification of the original one \cite{MIMURASUGATA}.

\begin{proof}
Following Mimura and Sugata, we consider the action $U\cdot X= UXU^T$ of $U(2n)$ on the manifold
$$M=\{X\in U(2n)\colon X+X^T=0\}.$$
This action turns out to be transitive with isotropy $Sp(n)$. Then $U(2n)/Sp(n)\cong M$. Here the symplectic group is identified (via complexification) with the subgroup of matrices $U\in U(2n)$ such that $UJU^T=J$, where
 $$J=
 \pmatrix{
  0 & -I_n \cr
  I_n&0\cr
}.
 $$
 
Now, consider the manifold 
$$M^\prime=JM=\{Y\in U(2n)\colon Y^T=-JYJ\}$$
which is diffeomorphic to $M$ because $J^2=-I$.

Let $Y\in M^\prime$, let $\lambda\in\C$ with $\vert \lambda \vert=1$. Then $Y\in\Omega(\lambda I)$ if and only if   $-\lambda$ is an eigenvalue of $Y$. But as remarked in \cite{MIMURASUGATA}, $Yv=-v\lambda$ implies $Y(J\bar v)=-(J\bar v)\lambda$, which means that  $Y$ can be diagonalized as
$$Y=U\pmatrix{D&0\cr 0&D\cr}U^*, \quad D=\diag(\lambda_1,\dots,\lambda_n),$$
so
$$Y+\lambda I = U\pmatrix{D+\lambda I&0\cr 0&D+\lambda I\cr}U^*,$$
showing that the maximum number of different eigenvalues of $Y$ is $n$.  This implies that when taking $n+1$ different complex numbers $z_0,\dots,z_n$, with $\vert z_k\vert =1$,  the open sets $\Omega(-z_k\, I)$ will cover $M^\prime$. 

It only remains to show that the  Cayley contraction remains inside $M^\prime$. More explicitly, for any $z\in\C$, $\vert z \vert=1$, let the Cayley map be
$$c_{z\,I}\colon \Omega (z\, I) \to T_{\bar z\, I}G,$$ 
let $Y\in \Omega(z\,I)\cap M^\prime$ and take the radial contraction $tc_{z\,I}(Y)$, $t\in[0,1]$. In the same way of Proposition \ref{PROP2} it is  easy to prove that 
$$c_{z\,I}(Y^T)=-c_{z\,I}(Y)^T$$ 
and that 
$$c_{z\,J}(-JYJ)=-Jc_{z\,I}(Y)J.$$
This implies that the image $c_{\bar z\,I}(tc_{z\,I}(Y)$ of the contraction by the inverse Cayley map is contained in $M^\prime$. 

Hence $\cat U(2n)/Sp(n)\leq n$.
On the other hand (\cite[p.~149]{MIMURATODA}), $$H(U(2n)/Sp(n))=\Lambda(x_1,x_5,x_9,\dots,x_{4n-3})$$ so $n=\lcp\leq \cat U(2n)/Sp(n)$. Equality follows.
\qed\end{proof}

\begin{theorem}  $\cat U(n)/O(n)=n.$ 
\end{theorem}

The proof  is completely analogous to the preceding one.

\subsection{The symplectic group $ Sp(2)$}\label{SP2}
The LS category of $G= Sp(2)$ was computed for the first time by P.~Schweitzer \cite{PAUL}, who proved that $\cat{Sp(2)}=3$.

 By using the Morse theory explained above we shall easily obtain an explicit covering by four categorical open sets. 

Let us consider the four critical points of a height  function $h_X$ as in Example \ref{ptos criticos diagonales} of Section \ref{MORSE}, namely: the identity $I=\diag(1,1)$, $P=\diag(-1,1)$, $-P$ and $-I$.
\begin{theorem}\label{scat de Sp2}
$ \{\pm\Omega_G(I), \pm\Omega_G( P)\}$ is a categorical covering of $G= Sp(2)$.
\end{theorem}
\begin{proof} 
Remember that $X\in\Omega_G(A)$ means that $A+X$ is invertible.
First we see that the open sets $\Omega_G(I)$ and $\Omega_G(-I)$  cover   the whole  group excepting the orbit $UPU^*$  of the  matrix $P$ by the adjoint action. In fact, by diagonalization, this orbit is formed by the matrices having $1$ and $-1$ as eigenvalues and is diffeomorphic to the sphere $S^4$, so it can be covered by stereographic projection.
 
Explicitly, we must prove that given $X=(x_{ij})=UPU^*$ in the orbit of $P$, either $P+X$ or 
$-P+X$ is invertible. 

Since $P^2=X^2=I$ we have 
 $$(P+X)^2=2 I+PX+XP=2\diag(1-x_{11},1+x_{22}).$$

Then $P+X$ is invertible if and only if  $x_{11}\neq 1$ and $x_{22}\neq-1$. Suppose $x_{11}=1$. The condition $X^*X=I$ means that the columns of $X$ form an orthonormal basis of $\HH^2$ for the hermitian product $\langle v,w\rangle=v^*w$.
 Then, $x_{11}=1$ implies $x_{12}=x_{21}=0$. But since $X$ is in the orbit of $P$, it must be $x_{22}=-1$. Hence $X=-P$. The same conclusion is obtained from  $x_{22}=-1$. So in fact $\Omega_G(P)$ covers all the orbit of $P$, excepting $-P$. Since $-P	\in\Omega_G(-P)$, the proof is done.
\qed\end{proof} 

\begin{remark} The cohomology of the symplectic group   is  
\cite[p.~119]{MIMURATODA}
$$H( Sp(n))=\Lambda (x_3,x_7,\ldots,x_{4n-1})$$
so the longest non-null product  is $x_3\wedge x_7\wedge \cdots \wedge x_{4n-1}$
and $\lcp Sp(n)=n$. However Schweitzer \cite{PAUL} was able to prove that $\cat  Sp(n)\geq n+1$ for $n\geq 2.$
\end{remark}

\subsection{Left eigenvalues of symplectic matrices}\label{IZQUIERDA}
Obviously it is not worthy to apply the method of critical points  above to the symplectic group $ Sp(n)$, $n>2$ (for instance it is known that $\cat{Sp(3)=5}$ \cite{TATO}). Instead, in this section we discuss the possibility of extending the eigenvalue method of paragraphs \ref{UNITARY} and \ref{SYMMETRIC} to the symplectic setting.

We shall briefly explain the underlying difficulties.
First, it is necessary to endow the quaternionic space $\HH^n$ with the structure of a {\em right} $\HH$-vector space, in order to obtain the usual results for the matrix associated to a linear map. Second, the theory of right eigenvalues is well established, including diagonalization of symplectic matrices \cite{BAKER,BRENNER,FAREPIDK,ZHANG1}. However, for a matrix $A\in  Sp(n)$ and a quaternion $\sigma\in\HH$, the condition that $A-\sigma I$ be invertible is {\em not} related to $\sigma$ being a right eigenvalue. Instead we must consider  {\em left} eigenvalues.

\begin{definition} A quaternion $\sigma\in\HH$ is a {\em left eigenvalue} of the matrix $A\in \mathcal{M}(n,\HH)$ if and only if there exists $v\in \HH^n$, $v\neq 0$, such that $Av=\sigma v$.
\end{definition}

Unfortunately, very little is known about left eigenvalues of quaternionic matrices. Their existence, number, and methods for computing them are only partially understood, see Zhang's paper \cite{ZHANG2} for a recent review.   For $n=2$ two of the authors were able to prove the following Theorem, based on previous results by Huang and So \cite{HS}.

\begin{theorem}[\cite{MP}]\label{INFINITE} A symplectic matrix $A\in  Sp(2)$ has either one, two or infinite left eigenvalues. The latter case can only occur when
$$A=L_q\circ R_\theta=\left(\matrix{
q\cos\theta & -q\sin\theta\cr
q\sin\theta & q\cos\theta\cr
}\right), \quad q\in\HH, \vert q \vert =1, \theta\in\R, \sin\theta\neq 0.$$
\end{theorem}

As a corollary it is possible to obtain the following deceptive result.

\begin{corollary}[\cite{MP3}]\label{CUATRONO} Let $\sigma_0,\dots,\sigma_3$ be four  arbitrary quaternions of module $1$. Then the categorical open sets 
$\Omega_G(\sigma_0\cdot I),\dots,\Omega_G(\sigma_3\cdot I)$ 
do {\em not} cover the group 
$G=\mathit{Sp(2)}$.
\end{corollary}

%
%

\begin{acknowledgements}
We acknowledge many useful discussions with D. Tanr\'e.
\end{acknowledgements}



\end{document}